\documentclass[11pt,reqno]{amsproc}

\usepackage{amsmath, amsthm, amscd, amsfonts, amssymb, graphicx, color}
\usepackage[bookmarksnumbered, plainpages]{hyperref}
\usepackage{graphicx}
\usepackage{epsfig}
 \newtheorem{thm}{Theorem}[section]
 \newtheorem{cor}[thm]{Corollary}
  
 \newtheorem{lem}[thm]{Lemma}
 
 \theoremstyle{definition}
 \newtheorem{defn}[thm]{Definition}
 \theoremstyle{remark}
 
 \numberwithin{equation}{section}
 
  \theoremstyle{Example}
 \newtheorem{ex}[thm]{Example}

 \newcommand{\A}{\mathcal{A}}
   \newcommand{\B}{\mathcal{B}}

\begin{document}

\setcounter{page}{1}

\renewcommand{\currentvolume}{00}

\renewcommand{\currentyear}{0000}

\renewcommand{\currentissue}{0}

\title[APPROXIMATE WEAK AMENABILITY OF CERTAIN BANACH ALGEBRAS]
 {APPROXIMATE WEAK AMENABILITY OF CERTAIN BANACH ALGEBRAS}

\author[B. SHOJAEE]{BEHROUZ SHOJAEE}

\address{ Department of Mathematics, Karaj Branch, Islamic Azad university, Karaj, Iran}
\email{shoujaei@kiau.ac.ir}

 \author[A. Bodaghi]{ABASALT BODAGHI}
\address{Department of Mathematics, Garmsar Branch, Islamic Azad University, Garmsar,
 Iran}
\email{abasalt.bodaghi@gmail.com}

\thanks{}



\date{}

\dedicatory{}


\begin{abstract}It is shown that for a locally compact group $G$, if $L^{1}(G)^{**}$ is approximately weakly amenable, then $M(G)$ is approximately weakly amenable. Then, new notions of approximate weak amenability and approximate cyclic amenability for
Banach algebras are introduced. Bounded $\omega^{*}$-approximately weakly [cyclic] amenable $\ell^{1}$-Munn algebras are characterized.

\end{abstract}
\maketitle

\section{Introduction}
The notion of weak amenability was introduced by Bade, Curtis and  Dales in \cite{bad} for commutative Banach algebras. Later, Johnson defined weak
amenability for arbitrary Banach algebras \cite{joh2} and showed that for a locally
compact $G$, the group algebra $L^{1}(G)$ is  weakly amenable (for shorter proof see \cite{dgh}). It is shown in \cite{glw} that if $L^{1}(G)^{**}$ is weakly amenable, then $M(G)$, the measure algebra of $G$ is weakly amenable. It is also proved
in \cite{dgh} that $M(G)$ is amenable if and only if the group $G$ is discrete and amenable. The notion of cyclic amenability for Banach algebras was introduced by Gr\o nb\ae k in \cite{gro2}. Then the approximate version of mentioned notions are studied in \cite{glo} and \cite{shb} for Banach algebras.

In \cite{ess}, Esslamzadeh introduced $\ell^{1}$-Munn algebras which are a class of Banach algebras. He investigated some basic
facts about the structure of $\ell^{1}$-Munn algebras and characterized those with bounded approximate identities. The characterizing of amenable $\ell^{1}$-Munn
algebras by explicit construction of approximate diagonals is also given there. In \cite{sho}, Shojaee et al. studied weak and cyclic amenability
of $\ell^{1}$-Munn algebras and showed that under certain condition, cyclic [resp. weakly] amenability of  a $\ell^{1}$-Munn algebra
is equivalent to the cyclic [resp. weakly] amenability of the underlying Banach algebra $\mathcal{A}$.

In Section 2 of this paper,  we show that if $\mathcal A^{**}$, the second dual of a Banach algebra $\mathcal A$ is approximately
weakly amenable then  $\mathcal A$ is essential. This could be regarded as the approximate version of a result of Ghahramani
and Laali \cite[Proposition 2.1]{gla}. We investigate some relationships between approximate weak amenability of Banach algebras
$\mathcal{A}$, $\mathcal{B}$ and the tensor product $\mathcal{A}\widehat{\otimes}\mathcal{B}$. The main result of this section is Theorem \ref{t3}
which asserts that for a locally compact $G$, approximate weak amenability of $L^{1}(G)^{**}$  implies approximate weak amenability
of $M(G)$. In Section 3, we introduce the concepts of bounded $\omega^{*}$-approximate weak [cyclic] amenability for Banach algebras.
By means of some examples, we show that these concepts are weaker than the weak and cyclic amenability. We also indicate some properties of such
Banach algebras. Finally,  we characterize  $\ell^{1}$-Munn algebras that are bounded $\omega^{*}$-approximately weakly [cyclic] amenable.

\section{Approximate weak amenability}

We first recall some definitions in the Banach algebras setting. Let $\mathcal A$ be a Banach algebra, and let $X$ be a Banach $\mathcal A$-bimodule.
A bounded linear map $D: \mathcal A \longrightarrow X$ is called a derivation if
$$D(ab)=D(a)\cdot b+a\cdot D(b) \qquad (a,b \in \mathcal A).$$\\
For each $x\in X$, we define a map ${\rm{ad}}_x: \mathcal A \longrightarrow X$
by ${\rm{ad}}_x(a)=a\cdot x-x\cdot a\,\, (a\in \mathcal A).$ It is easily seen that ${\rm{ad}}_x$ is a derivation. Derivations of this form are
called inner derivations. A derivation $D: \mathcal A
\longrightarrow X$ is said to be approximately inner if there exists a net $(x_i)\subseteq X$ such that
$$D(a)=\lim_i(a\cdot x_i-x_i\cdot a)\qquad (a \in \mathcal A).$$\\
Hence $D$ is approximately inner if it is in the closure of the
set of inner derivations with respect to the strong operator
topology on $B(\mathcal A)$, the space of bounded linear operators on $\mathcal A$. The Banach algebra $\mathcal A$ is approximately amenable
if every bounded derivation $D: \mathcal A \longrightarrow X^*$  is
approximately inner, for each Banach $\mathcal A$-bimodule $X$ \cite{glo}, where $X^*$ denotes the first dual space of $X$ which is a Banach
$\mathcal A$-bimodule in the canonical way. A Banach algebra $\mathcal{A}$ is called weakly amenable if every derivation
$D: \mathcal A \longrightarrow \mathcal A^*$ is inner and it is called approximately weakly amenable, if any such derivation
is approximately inner. $\mathcal A$ is called cyclic amenable if every cyclic derivation
from $\mathcal{A}$ into $\mathcal{A}^{*}$ (i.e., $\langle D(a),b\rangle+\langle D(b),a\rangle=0$, for all $a,b\in\A$) is inner (see \cite{sho}.

Let $ \square$ and $ \lozenge$ be the first and second Arens
products on the second dual space $\mathcal A^{**}$, then
$\mathcal A^{**}$ is a Banach algebra with respect to both of
these products.  Let
$Z_1(\mathcal A^{**})$ denote the first topological center
of $\mathcal A^{**}$, that is
$$Z_1(\mathcal A^{**})=\{a^{**} \in \mathcal A^{**} :b^{**}
\mapsto a^{**} \square b^{**} \hspace{0.2cm} \text{is} \hspace{0.2cm}
 \sigma (\mathcal A^{**},\mathcal A^*)\text{-continuous} \}.$$

The second topological centre is defined by
$$Z_2(\mathcal A^{**})=\{a^{**} \in \mathcal A^{**} :b^{**}
\mapsto b^{**} \lozenge a^{**} \hspace{0.2cm} \text{is} \hspace{0.2cm}
 \sigma (\mathcal A^{**},\mathcal A^*)\text{-continuous} \}.$$

\begin{thm}\label{tht}Let $\mathcal{A}$ be a Banach algebra.
\begin{enumerate}
\item[$\text{(i)}$] {Suppose that $\bf{B}$ is a closed subalgebra of $(\mathcal{A}^{**},\square)$ such that
$\mathcal{A}\subseteq\bf{B}$. If $\bf{B}$ is approximately amenable, then $\mathcal{A}$ is approximately amenable;}
\item[$\text{(ii)}$] {Suppose that $Z_1(A^{**})$ \emph{(}or $Z_2(A^{**})$\emph{)} is approximately amenable.
Then $\mathcal{A}$ is approximately  amenable.}
\end{enumerate}
\end{thm}
\begin{proof}
(i) Assume that $D:\mathcal{A}\longrightarrow \mathcal{X^{*}}$ is a continuous derivation. By \cite[Proposition 2.7.17(i)]{dal}
the map $D^{**}:(\mathcal{A}^{**},\square)\longrightarrow \mathcal{X}^{***}$ is a continuous derivation, and so $D^{**}|_{\bf{B}}$ is a  derivation.
Thus there exists a net $(x_{\alpha}^{***})\subseteq \mathcal{X}^{***}$ such that
$$D^{**}(b)=\lim_{\alpha}b\cdot x_{\alpha}^{***}-x_{\alpha}^{***}\cdot b\quad \quad (b\in \bf{B}).$$
Consider the projection map $P:\mathcal{X}^{***}\longrightarrow \mathcal{X}^{*}$ which is an $\mathcal A$-module. Then
$$D(a)=P(D^{**}(a))=\lim_{\alpha}a\cdot P(x_{\alpha}^{***})-P(x_{\alpha}^{***})\cdot a\quad(a\in \mathcal{A}).$$

Therefore $\mathcal{A}$ is approximately amenable.

(ii) It is immediately follows from (i).
\end{proof}

One should remember that the amenability case of Theorem \ref{tht} has been proved by Ghahramani and Laali in \cite[Proposition 1.1]{gla},
but our proof is different.

Recall that a topological algebra $\mathcal A$ is said to be essential if $\mathcal A^2$ is 
dense in $\mathcal A$. In \cite[Proposition 2.1]{esh}, Esslamzadeh and Shojaee proved that
if the Banach algebra $\mathcal{A}$ is approximately weakly
amenable, then $\mathcal{A}$ is essential. The same
conclusion holds if $\mathcal{A}^{**}$ is weakly amenable
\cite[Proposition 2.1]{gla}. We show this result for the approximate
case as follows.
\begin{thm} Let $\mathcal{A}$ be a Banach algebra. If \emph{(}$\mathcal{A}^{**},\square$\emph{)} is approximately weakly amenable, then $\mathcal{A}$ is essential.
\end{thm}

\begin{proof}Assume towards a contradiction that $\mathcal{A}^{2}$ is not dense in $\mathcal{A}$.
Take $a_{0}\in \mathcal{A}\backslash\mathcal{A}^{2}$ and  $\lambda\in \mathcal{A}^{*}$ such that
$\lambda|_{\mathcal{A}^{2}}=0$ and $\langle \lambda,a_{0}\rangle=1$. Consider the map $D:\mathcal{A}^{**}\longrightarrow \mathcal{A}^{***}$;
$a^{**}\mapsto\langle \lambda,a^{**}\rangle\lambda$. Obviously, $D$ is continuous and linear. For each $a^{**}, b^{**}\in \mathcal{A}^{**}$,
there are nets $(a_{\alpha}), (b_{\beta})\subseteq \mathcal{A}$ such that $a^{**}\square b^{**}=\omega^{*}-\lim_{\alpha}\lim_{\beta}a_{\alpha}b_{\beta}$.
We have
$\langle a^{**}\square b^{**}, \lambda\rangle=\lim_{\alpha}\lim_{\beta}\langle \lambda,a_{\alpha}b_{\beta}\rangle=0$, and so $D(a^{**}\square b^{**})=0$.
On the other hand,
\begin{align*}
\langle a^{**}\cdot D(b^{**}),c^{**}\rangle+\langle D(a^{**})\cdot b^{**},c^{**}\rangle
&=\langle D(b^{**}),c^{**}\square a^{**}\rangle+\langle D(a^{**}),b^{**}\square c^{**}\rangle\\
&=\langle b^{**},\lambda\rangle\langle c^{**}\square a^{**},\lambda\rangle\\
&+\langle a^{**},\lambda\rangle\langle b^{**}\square c^{**},\lambda\rangle=0.
\end{align*}
Thus $D:\mathcal{A}^{**}\longrightarrow\mathcal{A}^{***}$ is a derivation, but it is not approximately inner.
In fact  $\langle D(a_{0}),a_{0}\rangle=1$, whereas $$\lim_{\alpha}\langle ad_{\lambda_{\alpha}}(a_{0}),a_{0}\rangle
=\lim_{\alpha}\langle a_0\cdot\lambda_{\alpha}-\lambda_{\alpha}\cdot a_{0},a_{0}\rangle=\lim_{\alpha}\langle \lambda_{\alpha},
a_{0}^2-a_{0}^2\rangle=0,$$    for any net $(\lambda_{\alpha})\subseteq\mathcal{A}^{***}$. This being a contradiction of
$\mathcal{A}^{**}$ is approximately weakly amenable.
\end{proof}
Recall that a character on the Banach algebra $\mathcal{A}$ is a non-zero homomorphism from $\mathcal{A}$ into $\Bbb C$.
The set of characters on $\mathcal{A}$ is called the character space of $\mathcal{A}$ and denoted by $\Phi_{\mathcal{A}}$.
Also, $\mathcal A$ is said to be dual if there is a closed
submodule $\mathcal A_{*}$ of $\mathcal A^*$ such that $\mathcal A=\mathcal A_{*}^*$.

It is shown in part (ii) of \cite[Propositions 2.1]{esh} that the homomorphic image of an approximately weakly amenable
commutative Banach algebra is again approximately weakly amenable. In the next theorem, we generalize this result.

\begin{thm}\label{t1} Let $\mathcal{A}$ and $\mathcal{B}$ be
Banach algebras.
\begin{enumerate}
\item[$\text{(i)}$] {Suppose that $\varphi:\mathcal{A}\longrightarrow
\mathcal{B}$ and $\psi:\mathcal{B}\longrightarrow \mathcal{A}$
are continuous homomorphisms such that $\varphi\circ\psi=I_{\mathcal{B}}$. If $\mathcal{A}$ is approximately weakly amenable,
then $\mathcal{B}$ is approximately weakly amenable. Moreover, if $(\mathcal{A}^{**}, \square)$ is approximately weakly amenable,
then $(\mathcal{B}^{**}, \square)$ is approximately weakly amenable;}
\item[$\text{(ii)}$] {Suppose that $\mathcal{A}$ is a dual Banach
algebra. If $(\mathcal{A}^{**}, \square)$ is approximately weakly
amenable then $\mathcal{A}$ is approximately weakly amenable;}
\item[$\text{(iii)}$] {Suppose that $\mathcal{A}$ is commutative. Then $\mathcal{A}$ and $\mathcal{B}$ are approximately
weakly amenable if and only if the $\ell^1$-direct sum $\mathcal{A}\oplus_{1}\mathcal{B}$ is approximately weakly amenable;}
\item[$\text{(iv)}$] {Suppose that $\mathcal{A}$ is weakly amenable. Then $\mathcal{B}$ is approximately weakly amenable if and only
if the $\ell^1$-direct sum $\mathcal{A}\oplus_{1}\mathcal{B}$ is
approximately weakly amenable;}
\item[$\text{(v)}$] {Suppose that $\mathcal{A}$ is commutative. Then $\mathcal{A}$ is approximately weakly amenable if and only
if $\mathcal{A}\widehat{\otimes}\mathcal{A}$ is approximately weakly amenable;}
\item[$\text{(vi)}$] {Suppose that $\mathcal{A}$ and $\mathcal{B}$ are unital, and $\phi_{1}\in\Phi_{\mathcal{A}}$,
$\phi_{2}\in\Phi_{\mathcal{B}}$. If $\mathcal{A}\widehat{\otimes}\mathcal{B}$ is approximately weakly amenable, then
$\mathcal{A}$ and $\mathcal{B}$ are approximately weakly amenable.}
\end{enumerate}
\end{thm}
\begin{proof}(i) Let $D:\mathcal{B}\longrightarrow
\mathcal{B}^{*}$ be a derivation. We can consider $\B$ as an $\A$-bimodule with actions $a\cdot x=\varphi(a)x$ and $x\cdot a=x\varphi(a)$ for every
$a\in\A, x\in\B$. Hence the map $\varphi^{*}$ is an
$\mathcal{A}$-module homomorphism, and thus
\begin{align*}
\varphi^{*}\circ D\circ\varphi(ab) &=\varphi^{*}(D(\varphi(a))\cdot\varphi(b)+\varphi(a)\cdot D(\varphi(b)))\\
&=\varphi^{*}\circ D\circ\varphi(a)\cdot b+a\cdot\varphi^{*}\circ
D\circ\varphi(b),
\end{align*}
for all $a,b\in \mathcal{A}$. Hence, $\varphi^{*}\circ
D\circ\varphi:\mathcal{A}\longrightarrow \mathcal{A}^{*}$ is a
continuous derivation. Therefore there exists a net
$(a_{\alpha}^{*})\subseteq \mathcal{A}^{*}$ such that
$\varphi^{*}\circ D\circ\varphi(a)=\lim_{\alpha}(a\cdot
a_{\alpha}^{*}-a_{\alpha}^{*}\cdot a)\quad(a\in \mathcal{A}).$ The
equality $\varphi\circ\psi=I_{\mathcal{B}}$ implies
$\psi^{*}\circ\varphi^{*}=I_{\mathcal{B}^{*}}$, and thus for every
$c\in \mathcal{B}$, we get
\begin{align*}
D(c) &=\psi^{*}\circ\varphi^{*}\circ D\circ \varphi\circ\psi(c)\\
&=\psi^{*}(\varphi^{*}\circ D\circ\varphi(\psi(c))\\
&=\psi^{*}(\lim_{\alpha}(\psi(c)\cdot a_{\alpha}^{*}-a_{\alpha}^{*}\cdot \psi(c)))\\
&=\lim_{\alpha}\psi^{*}(\psi(c)\cdot a_{\alpha}^{*}-a_{\alpha}^{*}\cdot\psi(c))\\
&=\lim_{\alpha}(c\cdot
\psi^{*}(a_{\alpha}^{*})-\psi^{*}(a_{\alpha}^{*})\cdot c).
\end{align*}
The above equalities show that $\mathcal{B}$ is approximately weakly amenable. Since $\varphi^{**}\circ\psi^{**}=I_{\mathcal{B}^{**}}$,
$(\mathcal{B}^{**}, \square)$ is approximately weakly amenable.

(ii) According to \cite[Theorem 2.15]{dau}, $(\mathcal{A}_{*})^{\bot}$ is a $\omega^{*}$-closed ideal in
$\mathcal{A}^{**}$ and $\mathcal{A}^{**}=\mathcal{A}\oplus(\mathcal{A}_{*})^{\bot}$. Now, if
$\varphi:\mathcal{A}^{**}\longrightarrow \mathcal{A}$ is the projection map and $\psi:\mathcal{A}\longrightarrow \mathcal{A}^{**}$
is the inclusion  map, then $\varphi\circ\psi=I_{\mathcal{A}}$, hence by part (i) we get the desired result.

(iii) It is known that weak amenability and approximate weak
amenability coincide for a commutative Banach algebra, and so we
deduce the sufficiency part by \cite[Proposition 2.2(iii)]{esh}.

Conversely, the maps $\varphi:\mathcal{A}\oplus \mathcal{B}\longrightarrow \mathcal{A}$;  $a\oplus b\mapsto a$ and
$\psi:\mathcal{A}\longrightarrow \mathcal{A}\oplus \mathcal{B}$; $ a\mapsto a\oplus 0$ are continuous homomorphisms
and $\varphi\circ \psi=I_{\mathcal{A}}$. By (i), $\mathcal{A}$ is approximately weakly amenable. Similarly for $\mathcal{B}$.

(iv) The proof is immediately by \cite[Proposition 2.2(iii)]{esh}
and part (i).

(v) The sufficiency part follows immediately from \cite[Proposition
2.6]{gro}. For the converse, in light of \cite[Corollary
2.8.70]{dal} we can suppose that $\mathcal{A}$ has an identity.
Consider the homomorphisms
$\varphi:\mathcal{A}\widehat{\otimes}\mathcal{A}\longrightarrow
\mathcal{A}$ defined by $\varphi(a\otimes b)=ab$ and
$\psi:\mathcal{A}\longrightarrow
\mathcal{A}\widehat{\otimes}\mathcal{A}$ by $\psi (a)=a\otimes
e_{\mathcal{A}}$, where $e_{\mathcal{A}}$ is the identity of
$\mathcal{A}$. Easily, $\varphi\circ\psi=I_{\mathcal{A}}$. Now, part
(i) shows that $\mathcal{A}$ is approximately weakly amenable.

(vi) One can check that the maps
$\varphi:\mathcal{A}\widehat{\otimes}\mathcal{B}\longrightarrow
\mathcal{A}$, $\varphi(a\otimes b)=\phi_{2}(b)a$ and
$\psi:\mathcal{A}\longrightarrow
\mathcal{A}\widehat{\otimes}\mathcal{B}$, $\psi(a)=a\otimes
e_{\mathcal{B}}$ are homomorphisms so that
$\varphi\circ\varphi=I_{\mathcal{A}}$. It is a consequence of  part
(i) that $\mathcal{A}$ is approximately weakly amenable. Similarly
for $\mathcal{B}$.
\end{proof}

Recall that a linear functional
$d$ on $\mathcal{A}$ is a point derivation at $\varphi\in\Phi_{\mathcal{A}}$ if
$$d(ab)=\varphi(a)d(b)+\varphi(b)d(a)\quad (a,b\in \mathcal{A}).$$

\begin{thm}\label{tthh1}
Let $\mathcal{A}$ and $\mathcal{B}$ be Banach algebras and $\mathcal{A}\widehat{\otimes}\mathcal{B}$ is approximately weakly amenable.
\begin{enumerate}
\item[$\text{(i)}$] {Then both $\mathcal{A}$ and $\mathcal{B}$ are essential;}
\item[$\text{(ii)}$] {If $\varphi\in\Phi_{\mathcal{A}}$ and $\psi\in\Phi_{\mathcal{B}}$, then there are no non-zero
point derivations on both  $\mathcal{A}$ and $\mathcal{B}$;}
\item[$\text{(iii)}$] {If $\mathcal{A}\widehat{\otimes}\mathcal{A}$ is approximately weakly amenable,
then $\mathcal{A}$ is essential and there is no non-zero point derivation on $\mathcal{A}$.}
\end{enumerate}
\end{thm}
\begin{proof}(i) It suffices to consider $\mathcal{A}$. For $\mathcal{B}$ is similar. Suppose that
$\overline{\mathcal{A}^{2}}\neq \mathcal{A}$. Take $a_{0}\in\mathcal{A} \backslash\mathcal{A}^{2}$ and
$\lambda\in\mathcal{A}^{*}$ such that $\lambda|_{\mathcal{A}^{2}}=0$ and $\langle \lambda,a_{0}\rangle=1$.
Also, choose $\mu\in\mathcal{B}^{*}$ and  $b_{0}\in\mathcal{B}$ such that $\langle\mu,b_{0}\rangle=1$.
Define $D:\mathcal{A}\widehat{\otimes}\mathcal{B}\longrightarrow (\mathcal{A}\widehat{\otimes}\mathcal{B})^{*}$ by
$D(a\otimes b)=\langle \lambda,a\rangle\langle \mu,b\rangle(\lambda\otimes\mu)$ where $a\in \mathcal{A}$ and
$b\in \mathcal{B}$. It is easy to see that $D$ is a derivation. Due to approximate weak amenability of
$\mathcal{A}\widehat{\otimes}\mathcal{B}$, there exists a net $(x_{\alpha})\subseteq (\mathcal{A}\widehat{\otimes}\mathcal{B})^{*}$
such that $D(a\otimes b)=\lim_{\alpha}(a\otimes b)\cdot x_{\alpha}-x_{\alpha}\cdot(a\otimes b)\,\,(a\in \mathcal{A},b\in \mathcal{B})$.
Therefore  $\langle D(a_{0}\otimes b_{0}),a_{0}\otimes b_{0}\rangle=1$. On the other hand,
$\lim_{\alpha}(\langle(a_{0}\otimes b_{0})\cdot x_{\alpha}-x_{\alpha}\cdot(a_{0}\otimes b_{0}),a_{0}\otimes b_{0}\rangle)=0$.
This contradicts our assumption.

(ii) Suppose that $d$ is a
non-zero continuous point derivation at
$\varphi_{1}\in\Phi_{\mathcal{A}}$. We can show that  the map
$D:\mathcal{A}\widehat{\otimes}\mathcal{B}\longrightarrow(\mathcal{A}\widehat{\otimes}\mathcal{B})^{*}$,
$D(a\otimes b)=d(a)\psi(b)\varphi_{1}\otimes\psi$ is a derivation.
Since $\mathcal{A}\widehat{\otimes}\mathcal{B}$ is approximately
weakly amenable, there exists a net $(x_{\alpha})\subseteq
(\mathcal{A}\widehat{\otimes}\mathcal{B})^{*}$ such that
$$D(a\otimes b)=\lim_{\alpha}(a\otimes b)\cdot x_{\alpha}-x_{\alpha}\cdot(a\otimes b)\quad\quad (a\in\mathcal{A}, b\in\mathcal{B}).$$
Take $b\in\mathcal{B}$ such that $\psi(b)=1$. Then
$d(a)\varphi_{1}(a)=0$,  and so $d|_{(\A-ker\varphi_{1})}=0$. Thus, since $d$ is a point derivation at $\varphi_{1}$, and also using
(i) we obtain  $d=0$ which is a contradiction.

 (iii) The result is a direct consequence of parts (i) and (ii).
\end{proof}

\begin{thm}\label{t2} Let $G$ be a locally compact group. Then the
following are equivalent:
\begin{enumerate}
\item[$\text{(i)}$] {$M(G)$ is weakly amenable;}
\item[$\text{(ii)}$] {$M(G)$ is approximately weakly amenable;}
\item[$\text{(iii)}$] {There is no non-zero, continuous point derivation at a character of $M(G)$;}
\item[$\text{(iv)}$] {$G$ is discrete.}
\end{enumerate}
\end{thm}
\begin{proof} (i)$\Rightarrow$(ii) is obvious.

(ii)$\Rightarrow$(iii) By \cite[Proposition 2.1]{esh} and \cite[Proposition 1.3]{dgg}.

(iii)$\Rightarrow$(iv) By \cite[Theorem 3.2]{dgh}.

(iv)$\Rightarrow$(i) By \cite[Theorem 1.2]{dgh}.
\end{proof}

Let $G$ be a locally compact group. Recall that $LUC(G)$ is the
space of bounded left uniformly continuous functions on $G$ under
the supremum norm and $C_{0}(G)$ is the space of continuous
functions on $G$ vanishing at infinity.

\begin{thm}\label{t3} Let $G$ be a locally compact group, and let
$(L^{1}(G)^{**},\square)$ be approximately weakly amenable. Then $M(G)$ is
approximately weakly amenable.
\end{thm}
\begin{proof}  By \cite[Lemma 1.1]{gll}, $LUC(G)^{*}=M(G)\oplus C_{0}(G)^{\perp}$ where $C_{0}(G)^{\perp}$ is
a closed ideal in $LUC(G)^{*}$. Assume that $E$ is a right identity for $L^{1}(G)^{**}$ in which $\|E\|=1$.
Then $L^{1}(G)^{**}=EL^{1}(G)^{**}\oplus(1-E)L^{1}(G)^{**}$ for which $(1-E)L^{1}(G)^{**}$ is a closed ideal
in $L^{1}(G)^{**}$. In addition, by \cite{ghl}, $EL^{1}(G)^{**}=LUC(G)^{*}$. Therefore the projection
maps $P_{1}:L^{1}(G)^{**}\longrightarrow LUC(G)^{*}$, $P_{2}:LUC(G)^{*}\longrightarrow M(G)$ and the
inclusion maps $\iota_{1}:LUC(G)^{*}\longrightarrow L^{1}(G)^{**}$, $\iota_{2}:M(G)\longrightarrow LUC(G)^{*}$
are homomorphisms such that $P_{1}\circ \iota_{1}=I_{LUC(G)^{*}}$ and $P_{2}\circ \iota_{2}=I_{M(G)}$.
By assumption that $L^{1}(G)^{**}$ is approximately weakly amenable, the above relations and part (i) of
Theorem \ref{t1}, we deduce that $LUC(G)^{*}$ approximately weakly amenable. Consequently, $M(G)$ is approximately weakly amenable.
\end{proof}

\begin{cor} For a non-discrete locally compact group $G$,
the Banach algebra $(L^{1}(G)^{**},\Box)$ is not approximately
weakly amenable.
\end{cor}
\begin{proof} The result follows from Theorem \ref{t2} and Theorem \ref{t3}.
\end{proof}

\section{Bounded $\omega^{*}$-approximately weak [cyclic] amenability}
We first introduce  two new notions of amenability;
bounded $\omega^{*}$-approximate weak amenability and bounded
$\omega^{*}$-approximate cyclic amenability as follows:
\begin{defn}A Banach algebra $\A$ is bounded $\omega^{*}$-approximately weakly amenable if for every  continuous
 derivation $D:\A\longrightarrow \A^{*}$, there is a net
 $(x_{\alpha})\subseteq A^{*}$, such that the net $(ad_{x_{\alpha}})$ is norm bounded in $B(\A,\A^{*})$ and
 $$D(a)=\omega^{*}-\lim_{\alpha}ad_{x_{\alpha}}(a)\quad (a\in\A).$$
 \end{defn}
 \begin{defn}A Banach algebra $\A$ is bounded
$\omega^{*}$-approximately cyclic amenable if for every cyclic
continuous derivation $D:\A\longrightarrow \A^{*}$, there is a net
$(x_{\alpha})\subseteq \A^{*}$, such that the net $(ad_{x_{\alpha}})$
is norm bounded in $B(\A,\A^{*})$ and
$$D(a)=\omega^{*}-\lim_{\alpha}ad_{x_{\alpha}}(a)\quad (a\in\A).$$
 \end{defn}

Obviously that all notions of weak amenability, approximate weak
amenability and bounded $\omega^{*}$- approximate weak amenability
coincide for a commutative Banach algebra. Moreover, if $\mathcal A$
is a commutative Banach algebra without identity, then  it is
bounded $\omega^{*}$-approximately weakly amenable if and only if
$\A^{\#}$ is bounded $\omega^{*}$-approximately weakly amenable.
These facts fail to be true in general. In the following
example we give a Banach algebra that is bounded
$\omega^{*}$-approximately weakly [cyclic] amenable but  not weakly
[cyclic] amenable.

\begin{ex}\label{ex1} \emph{We are following Example 6.2 of \cite{ghl}. So
we have a Banach algebra $\A$ that is not weakly amenable, but is
approximately amenable. In other words, as is showed in the mentioned example
 for every derivation $D:\A\longrightarrow\A^{*}$ there
exists a sequence $(x_{n})\subseteq\A^{*}$ such that
$Da=\lim_{n}ad_{x_{n}}(a)$ for each $a\in\A$. Hence the sequence
$(ad_{x_{n}})$ is bounded and thus $\A$ is bounded
$\omega^{*}$-approximately weakly amenable.}
\end{ex}

It should be mentioned that some properties such as,  being essential, not having non-zero point derivation, hold for Banach
algebras that are bounded $\omega^{*}$-approximately weakly amenable hold. The proofs of them are similar to the [approximate]
weak amenability case. The following theorem is analogous to Theorem \ref{t1} and we prove only part (i).

\begin{thm}\label{t1tt} Let $\mathcal{A}$ and $\mathcal{B}$ be
Banach algebras.
\begin{enumerate}
\item[$\text{(i)}$] {Suppose that $\varphi:\mathcal{A}\longrightarrow
\mathcal{B}$ and $\psi:\mathcal{B}\longrightarrow \mathcal{A}$
are continuous homomorphisms such that $\varphi\circ\psi=I_{\mathcal{B}}$. If $\mathcal{A}$ is bounded
$\omega^{*}$-approximately weakly amenable are, then $\mathcal{B}$ so is. Moreover, if $(\mathcal{A}^{**}, \square)$ is
bounded $\omega^{*}$-approximately weakly amenable, then $(\mathcal{B}^{**}, \square)$ so is;}
\item[$\text{(ii)}$] {Suppose that $\mathcal{A}$ is a dual Banach
algebra. If $(\mathcal{A}^{**}, \square)$ is bounded $\omega^{*}$-approximately weakly amenable, then $\mathcal{A}$ so is;}
\item[$\text{(iii)}$] {Suppose that $\mathcal{A}$ is commutative. Then $\mathcal{A}$ and $\mathcal{B}$ are
bounded $\omega^{*}$-approximately weakly amenable  if and only if the $\ell^1$-direct sum
$\mathcal{A}\oplus_{1}\mathcal{B}$ is bounded $\omega^{*}$-approximately weakly amenable;}
\item[$\text{(iv)}$] {Suppose that $\mathcal{A}$ is weakly amenable. Then $\mathcal{B}$ is bounded $\omega^{*}$-approximately weakly
amenable if and only if the $\ell^1$-direct sum
$\mathcal{A}\oplus_{1}\mathcal{B}$ is bounded
$\omega^{*}$-approximately weakly amenable;}
\item[$\text{(v)}$] {Suppose that $\mathcal{A}$ is commutative. Then $\mathcal{A}$ is bounded $\omega^{*}$-approximately
weakly amenable if and only if $\mathcal{A}\widehat{\otimes}\mathcal{A}$ is bounded $\omega^{*}$-approximately weakly amenable;}
\item[$\text{(vi)}$] {Suppose that $\mathcal{A}$ and $\mathcal{B}$ are unital, and
$\phi_{1}\in\Phi_{\mathcal{A}}$, $\phi_{2}\in\Phi_{\mathcal{B}}$. If $\mathcal{A}\widehat{\otimes}\mathcal{B}$ is
bounded $\omega^{*}$-approximately weakly amenable, then $\mathcal{A}$ and $\mathcal{B}$ are bounded $\omega^{*}$-approximately weakly amenable .}
\end{enumerate}
\end{thm}
\begin{proof} We follow the argument of the proof of Theorem \ref{t1}. Let $D:\mathcal{B}\longrightarrow
\mathcal{B}^{*}$ be a derivation. Then $\varphi^{*}\circ D\circ\varphi:\mathcal{A}\longrightarrow \mathcal{A}^{*}$ is a
continuous derivation, and so there exists a net $(a_{\alpha}^{*})\subseteq\mathcal{A}^{*}$ such that the net $(ad_{a_{\alpha}^{*}})$ is bounded and
$$\varphi^{*}\circ D\circ\varphi(a)=\omega^{*}-\lim_{\alpha}ad_{a_{\alpha}^{*}}(a)=\omega^{*}-\lim_{\alpha}(a\cdot
a_{\alpha}^{*}-a_{\alpha}^{*}\cdot a)\quad(a\in \mathcal{A}).$$

We have $\psi^{*}\circ\varphi^{*}=I_{\mathcal{B}^{*}}$ and also $\psi^{*}$ is $\omega^{*}$-continuous. Thus for every $c\in \mathcal{B}$, we get
\begin{align*}
D(c) &=\psi^{*}\circ\varphi^{*}\circ D\circ \varphi\circ\psi(c)=\psi^{*}(\varphi^{*}\circ D\circ\varphi(\psi(c))\\
&=\omega^{*}-\lim_{\alpha}\psi^{*}(\psi(c)\cdot a_{\alpha}^{*}-a_{\alpha}^{*}\cdot\psi(c))\\
&=\omega^{*}-\lim_{\alpha}(c\cdot\psi^{*}(a_{\alpha}^{*})-\psi^{*}(a_{\alpha}^{*})\cdot c)\\
&=\omega^{*}-\lim_{\alpha}ad_{\psi^{*}(a_{\alpha}^{*})}(c).
\end{align*}

Since $ad_{\psi^{*}(a_{\alpha}^{*})}(c)=\psi^{*}(ad_{a_{\alpha}^{*}}(\psi(c)))$ and the net $(ad_{a_{\alpha}^{*}})$ is bounded,
the net $(ad_{\psi^{*}(a_{\alpha}^{*})})$ is bounded. The arguments of other parts in Theorem \ref{t1}  work to finish the proof.
\end{proof}

We should recall that a Banach algebra $\A$ is approximately
amenable if and only if $\A$ is $\omega^{*}$-approximately amenable
\cite[Theorem 2.1]{glz}. Now, in view of \cite[Theorem 3.2]{glo}, we can show
that a locally compact group $G$ is amenable if and only if $L^{1}(G)$ is bounded
[$\omega^{*}$]-approximately amenable.

Using Example \ref{ex1}, we present a Banach algebra which is
[bounded $\omega^{*}$-] approximately weakly amenable but it is
neither weakly amenable nor approximately amenable.

\begin{ex}\emph{Let $G$ be a non-amenable discrete group. Then $\ell^{1}(G)$ is approximately weakly amenable,
but not approximately amenable \cite[Theorem 3.2]{glo}. Now,
consider the Banach algebra $\A$ as in Example \ref{ex1}. Then
$\B=\A\oplus_{1}\ell^{1}(G)$ equipped with $\ell^{1}$-norm is a
Banach algebra. The maps $\varphi_{1}:\B\longrightarrow\A$ and
$\varphi_{2}:\A\longrightarrow\B$ are homomorphisms in which
$\varphi_{1}\circ\varphi_{2}=I_{\A}$. Since $\ell^{1}(G)$ is weakly
amenable, $\B$ is [bounded $\omega^{*}$-] approximately weakly
amenable by \cite[Proposition 2.2(iii)]{esh} and the part (iv) of Theorem \ref{t1tt}. By Theorem \ref{t1} (i) and \cite[Proposition 2.2]{glo}, $\B$
can not be  weakly amenable nor approximately amenable.}
\end{ex}

In light of Theorem \ref{t1tt}, we can prove Theorems \ref{tthh1}, \ref{t2} and  \ref{t3} for the bounded $\omega^{*}-$approximate weak amenability case.

Let $\A$ be a unital Banach algebra, $I$ and $J$ be arbitrary index sets and $P$ be a $J\times I$ matrix over $\A$ such that all
of its non-zero entries are invertible and $\|P\|_{\infty}=\text{sup}\{\|P_{ji}\|: \, j\in J, i\in I\}\leq 1$.
The set $\ell^{1}(I\times J,\A)$, the vector space of all $I\times J$ matrics $X$ over $\mathcal A$ with product
$X\circ Y=XPY$ is a Banach algebra that we call the $\ell^{1}$-{\it Munn-algebra} over $\A$ with sandwich matrix $P$
and denote it by $LM(\A,P)$ (for more information see \cite{ess}).

\begin{thm}\label{t4}
If $\A$ is bounded $\omega^{*}$-approximately
cyclic amenable Banach algebra, then so is $LM(\A,P)$.\end{thm}
\begin{proof}
Suppose that $\beta\in J$,$\alpha\in I$ such that
$P_{\alpha\beta}\neq0$ and $q=P_{\alpha\beta}^{-1}$. Let $D :LM(\A,P)\longrightarrow LM(\A,P)^{*}$ be a bounded cyclic
derivation. Define $\widehat{D}$ via
$$\widehat{D}:\A\longrightarrow \A^{*},\quad \langle\widehat{D}a,b\rangle
=\langle
D(qa\varepsilon_{\beta\alpha}),qb\varepsilon_{\beta\alpha}\rangle,$$
for all $a,b\in\mathcal A$. Clearly, $\widehat{D}$ is a bounded linear map. By \cite[Theorem 2.1]{sho}, $\widehat{D}$
is a bounded cyclic derivation, hence there exists a net $(\widehat{\psi}_{\gamma})\subseteq\mathcal A$ such that the net
$(ad_{\widehat{\psi}_{\gamma}})$ is bounded and
$$\widehat{D}(a)=\omega^{*}-\lim_{\gamma}ad_{\widehat{\psi}_{\gamma}}(a)=
\omega^{*}-\lim_{\gamma}a\cdot\widehat{\psi}_{\gamma}-\widehat{\psi}_{\gamma}\cdot a.$$
Put $\psi_{\gamma}(a\varepsilon_{ij})=\widehat{\psi}_{\gamma}(p_{ji}a)+\langle D(q\varepsilon_{\beta j}),
 a\varepsilon_{i\alpha}\rangle$ for all $i\in I,j\in J,a\in\A$. It is easy to see that $\psi_{\gamma}\in LM(\A,P)^{*}.$ We wish to show that the net
$(ad_{\psi_{\gamma}})$ is bounded. For every
$a,b\in\A$, $j,l\in J$ and $i,k\in I$, we have
\begin{align*}
\langle ad_{\psi_{\gamma}}(a\varepsilon_{ij}),b\varepsilon_{kl}\rangle &=\langle
a\varepsilon_{ij}\cdot \psi_{\gamma}-\psi_{\gamma}\cdot a\varepsilon_{ij},
b\varepsilon_{kl}\rangle\\
&=\langle\psi_{\gamma}, b\varepsilon_{kl}\circ a\varepsilon_{ij}\rangle
-\langle\psi_{\gamma},a\varepsilon_{ij}\circ  b\varepsilon_{kl}\rangle\\
&=\langle\psi_{\gamma}, bP_{li}a\varepsilon_{kj}\rangle-\langle \psi_{\gamma},aP_{jk}b\varepsilon_{il}\rangle\\
&=\langle\widehat{\psi}_{\gamma}, p_{jk}bP_{li}a\rangle+\langle D(q\varepsilon_{\beta j}),bP_{li}a\varepsilon_{k\alpha}\rangle\\
&-\langle\widehat{\psi}_{\gamma},P_{li}aP_{jk}b\rangle-\langle D(q\varepsilon_{\beta j}),aP_{jk}b\varepsilon_{i\alpha}\rangle\\
&=\langle ad_{\widehat{\psi}_{\gamma}}(P_{li}a), P_{jk}b\rangle+\langle D(q\varepsilon_{\beta j}),
 bP_{li}a\varepsilon_{k\alpha}\rangle\\
 &-\langle D(q\varepsilon_{\beta j}),aP_{jk}b\varepsilon_{i\alpha}\rangle.
\end{align*}
Thus $|\langle ad_{\psi_{\gamma}}(a\varepsilon_{ij}),b\varepsilon_{kl}\rangle|\leq \|ad_{\widehat{\psi}_{\gamma}}\|\|a\|\|b\|
+2\|D\|\|a\|\|b\|$. On the other hand, the net $(ad_{\widehat{\psi}_{\gamma}})$ is bounded, so
$$|\langle ad_{\psi_{\gamma}}(B_{1}),B_{2}\rangle|\leq (\|ad_{\widehat{\psi}_{\gamma}}\|+2\|D\|)\|B_{1}\|_{1}\|B_{2}\|_{1}
\quad (B_{1},B_{2}\in LM(\A,P)).$$
Now, let $S=a\varepsilon_{ij}$ and $T=b\varepsilon_{kl}$ be non-zero
elements in $LM(\A,P)$ and $U=q\varepsilon_{\beta j}$,
$V=q\varepsilon_{\beta l}$, $X=a\varepsilon_{i\alpha}$ and
$Y=qp_{jk}b\varepsilon_{\beta\alpha}.$ Then, $S=X\circ U$, $U\circ
T=Y\circ V$ and $\langle D(V\circ X),Y\rangle+\langle DY,V\circ X\rangle=0$. By [18, Theorem 2.1] we get
\begin{equation}\label{ee1}\langle D(X),U\circ T\rangle=\langle D(V\circ
X),Y\rangle-\langle D(V),X\circ Y\rangle.\end{equation}
Also
\begin{eqnarray}\label{ee2}
 \nonumber
\langle D(V\circ X),Y\rangle &=& \langle D(qP_{li}a\varepsilon_{\beta\alpha}),qP_{jk}b\varepsilon_{\beta\alpha}\rangle\\
   \nonumber
   &=& \langle \widehat{D}(P_{li}a),P_{jk}b\rangle\\
   \nonumber
   &=&  \lim_{\gamma}\langle (P_{li}.a).\widehat{\psi}_{\gamma}-\widehat{\psi}_{\gamma}.(P_{li}a),P_{jk}b\rangle\\
   &=& \lim_{\gamma}(\langle\widehat{\psi}_{\gamma}, P_{jk}bP_{li}a\rangle-\langle\widehat{\psi}_{\gamma}, P_{li}aP_{jk}b\rangle).
\end{eqnarray}
 Applying (\ref{ee1}) and (\ref{ee2}), we have
 \begin{align*}
\langle D(S),T\rangle &=\langle D(U),T\circ X\rangle+\langle D(X),U\circ T\rangle\\
&=\langle D(U),T\circ X\rangle+\langle D(V\circ X),Y\rangle-\langle D(V),X\circ Y\rangle\\
&=\langle D(q\varepsilon_{\beta j}),bp_{li}a\varepsilon_{k\alpha}\rangle+\lim_{\gamma}(\langle\widehat{\psi}_{\gamma},p_{jk}bp_{li}a\rangle\\
&-\langle \widehat{\psi}_{\gamma},p_{li}ap_{jk}b\rangle)-\langle D(q\varepsilon_{\beta l}),ap_{jk}b\varepsilon_{i\alpha}\rangle\\
&=\lim_{\gamma}(\langle \psi_{\gamma},bp_{li}ap_{kj}\rangle-\langle \psi_{\gamma},ap_{jk}b\varepsilon_{il}\rangle)\\
&=\lim_{\gamma}(\langle \psi_{\gamma},T\circ S\rangle-\langle \psi_{\gamma},S\circ T\rangle)\\
&=\lim_{\gamma}\langle ad_{\psi_{\gamma}}(S),T\rangle.
\end{align*}
The net $(ad_{\psi_{\gamma}})$ is bounded, and thus
$$\langle D(B_{1}), B_{2}\rangle=\lim_{\gamma}\langle ad_{\psi_{\gamma}}B_{1}, B_{2}\rangle\quad (B_{1},B_{2}\in LM(A,P)).$$
The above equality shows that
$$D(B)=\omega^{*}-\lim_{\gamma}ad_{\psi_{\gamma}}B\quad (B\in
LM(A,P)).$$
Therefore $LM(A,P)$  is bounded $\omega^{*}$-approximately cyclic amenable.
\end{proof}
The proof of the following lemma is similar to the proof of \cite[Lemma 2.2]{sho}, so is omitted.
\begin{lem}\label{l1}
If $\A$ is bounded $\omega^{*}$-approximately
weakly amenable Banach algebra, then every continuous derivation
$D:LM(\A,P)\longrightarrow LM(\A,P)^{*}$ is cyclic.\end{lem}

The next theorem is an immediate consequence of Theorem \ref{t4} and
 Lemma \ref{l1}.
\begin{thm}\label{tthh}
If $\A$ is  bounded $\omega^{*}$-approximately weakly amenable, then so is $LM(\A,P).$
\end{thm}

In the upcoming theorem we show that the converse of Theorems \ref{t4} and \ref{tthh} are true as long as the sandwich matrix P
is square; i.e, the index sets $I$ and $J$ are equal \cite[Remark 2.4]{sho}.

\begin{thm}
Suppose $P$ is a regular square matrix and $LM(\A,P)$ has a bounded approximate identity. Then $\A$ is bounded $\omega^{*}$-approximately
weakly \emph{[}resp. cyclic\emph{]} amenable if and if $LM(\A,P)$ is bounded $\omega^{*}$-approximately weakly \emph{[}resp. cyclic\emph{]} amenable.
\end{thm}
\begin{proof}
By Theorem \ref{tthh} we need only to prove the converse statement. According to \cite{ess},
the index set $I$ is finite and $LM(\A,P)$ is topologically isomorphic
to $\A\widehat{\otimes}M_{n}$ for some $n\in \Bbb N$. If $D :\A\longrightarrow \A^{*}$ be a bounded derivation,
then $D\otimes1$ is a bounded derivation from $\A\widehat{\otimes}M_{n}$ to
$(A\widehat{\otimes}M_{n})^{*}$. Moreover, if $D$ is cyclic, then so is $D\otimes 1$. Thus there exists a
net $(X_{\alpha})\in (\A\widehat{\otimes}M_{n})^{*}$
such that the net $(ad_{X_{\alpha}})$ is bounded and $D(\B)=\omega^{*}$-$\lim_{\alpha}ad_{X_{\alpha}}(\B)$
for every $\B\in \A\widehat{\otimes}M_{n}$. For each $\alpha$, we put $X_{\alpha}=\Sigma_{i,j=1}^{n}x_{ij}^{\alpha}\otimes
\varepsilon_{ij}\quad (x_{ij}^{\alpha}\in A^{*}).$ Now, for $a,b\in A$ we get
\begin{align*}
D(a)\otimes\varepsilon_{11}&=(D\otimes 1)(a\otimes\varepsilon_{11})\\
&=\omega^{*}-\lim_{\alpha}((a\otimes\varepsilon_{11})(\Sigma_{i,j=1}^{n}x_{ij}^{\alpha}\otimes\varepsilon_{ij}
)-(\Sigma_{i,j=1}^{n}x_{ij}^{\alpha}\otimes\varepsilon_{ij})(a\otimes\varepsilon_{11}))\\
&=\omega^{*}-\lim_{\alpha}(\Sigma_{i=1}^{n}a.\cdot x_{i1}^{\alpha}\otimes\varepsilon_{i1}-
\Sigma_{j=1}^{n}x_{1j}^{\alpha}a\otimes\varepsilon_{1j}).
\end{align*}
 Thus
$$\langle D(a)\otimes\varepsilon_{11},b\otimes\varepsilon_{11}\rangle
=\lim_{\alpha}\langle\Sigma_{i=1}^{n}a\cdot x_{i1}^{\alpha}\otimes\varepsilon_{i1}-
\Sigma_{j=1}^{n}x_{1j}^{\alpha}\cdot a\otimes\varepsilon_{1j},b\otimes\varepsilon_{11}\rangle.$$
We have $\langle D(a),b\rangle=\lim_{\alpha}\langle a\cdot x_{11}^{\alpha}-x_{11}^{\alpha}\cdot a,b\rangle$, and so
$D(a)=\omega^{*}-\lim_{\alpha}(a\cdot x_{11}^{\alpha}-x_{11}^{\alpha}\cdot a)$.
To complete of the proof it is enough to show that the net $(ad_{x_{11}^{\alpha}})$ is bounded. For this, we have
\begin{align*}
\langle ad_{x_{11}^{\alpha}}(a),b\rangle &=\langle a\cdot x_{11}^{\alpha}-x_{11}^{\alpha}\cdot a,b\rangle\\
&=\langle\Sigma_{i=1}^{n}a\cdot x_{i1}^{\alpha}\otimes\varepsilon_{i1}-\Sigma_{j=1}^{n}x_{1j}^{\alpha}
\cdot a\otimes\varepsilon_{1j},b\otimes\varepsilon_{11}\rangle\\
&=\langle ad_{X_{\alpha}}(a\otimes\varepsilon_{11}),b\otimes\varepsilon_{11}\rangle.
\end{align*}
Hence
$$|\langle ad_{x_{11}^{\alpha}}(a),b\rangle|\leq\|ad_{X_{\alpha}}\|\|a\|\|b\|.$$
Therefore $\mathcal A$ is bounded $\omega^{*}$-approximately weakly [resp.
cyclic] amenable.\end{proof}

\section*{Acknowledgement}

The authors sincerely thank the anonymous reviewers for their careful reading, constructive comments and fruitful suggestions to improve the quality of
the first draft. The first author would like to thank the
Islamic Azad University of Karaj for its financial support.

\end{document}